\documentclass[leqno]{amsart}
\usepackage{amsmath}
\usepackage{amssymb}
\usepackage{amsthm}
\usepackage{enumerate}
\usepackage[mathscr]{eucal}
\theoremstyle{plain}
\newtheorem{theorem}{Theorem}[section]
\newtheorem{prop}[theorem]{Proposition}
\newtheorem{cor}{Corollary}[theorem]

\theoremstyle{definition}

\newtheorem{remark}{Remark}[section]

\newtheorem{example}{Example}[theorem]

\usepackage[pagewise]{lineno}
\begin{document}

\title[On the norm attainment set of a bounded linear operator]{On the norm attainment set of a bounded linear operator}
\author[Debmalya Sain]{Debmalya Sain}

\newcommand{\acr}{\newline\indent}

\address{\llap{\,}Department of Mathematics\acr
                              Jadavpur University\acr
                              Kolkata 700032\acr
                              West Bengal\acr
                              INDIA}
\email{saindebmalya@gmail.com}

\thanks{The author lovingly appreciates the enormous efforts on part of his mother, Mrs.~Asa Das(Sain), towards rightly shaping his philosophy. He feels extremely elated to acknowledge the incessant inspirations of Ms. Tilottama Majumder. He is grateful to Prof. Kallol Paul for his affectionate guidance. He would also like to thank Mr. Aniket Bhanja for pointing out a typo.} 


\begin{abstract}
In this paper we explore the properties of a bounded linear operator defined on a Banach space, in light of operator norm attainment. Using Birkhoff-James orthogonality techniques, we give a necessary condition for a bounded linear operator attaining norm at a particular point of the unit sphere. We prove a number of corollaries to establish the importance of our study. As part of our exploration, we also obtain a characterization of smooth Banach spaces in terms of operator norm attainment and Birkhoff-James orthogonality. Restricting our attention to $ l_{p}^{2} (p \in \mathbb{N}\setminus \{ 1 \})$ spaces, we obtain an upper bound for the number of points at which any linear operator, which is not a scalar multiple of an isometry, may attain norm. 

\end{abstract}

\maketitle

\section{Introduction.}
The principal aim of this paper is to explore the structure and properties of the norm attainment set of a bounded linear operator on a Banach space. Regarding the existential question of the norm attainment of a bounded linear operator, it is well known that a compact linear operator on a reflexive Banach space must attain norm at some point of the unit sphere. Furthermore, if the Banach space is strictly convex then any nonzero continuous linear functional defined on the space may attain maximum at at most one point of the unit sphere. However, to the best of our knowledge, no analogous result is available in the literature for bounded linear operators on Banach spaces. On the other hand, information regarding many important properties of a bounded linear operator, including smoothness of the operator \cite{P}, can be deduced from the norm attainment set of the operator. In the present paper, using Birkhoff-James orthogonality techniques, we strive to obtain a computable necessary condition for a  bounded linear operator on a Banach space to attain norm at a particular point of the unit sphere.\\   

Let us first fix our notations and terminologies. In this paper, letters $ \mathbb{X}, \mathbb{Y} $ denote Banach spaces. Throughout the paper, we consider the Banach spaces to be real. Let $ B_\mathbb{X}=\{x \in \mathbb{X} : \|x\| \leq 1\} $ and $ S_\mathbb{X}=\{x \in \mathbb{X} : \|x\|=1\} $ be the unit ball and the unit sphere of the Banach space $ \mathbb{X}$ respectively. Let $ \mathbb{L}(\mathbb{X},\mathbb{Y}) ( \mathbb{K}(\mathbb{X},\mathbb{Y}))$ denote the set of all bounded (compact) linear operators from the Banach space $ \mathbb{X} $ to another Banach space $ \mathbb{Y}. $ We write $ \mathbb{L}(\mathbb{X},\mathbb{Y}) = \mathbb{L}(\mathbb{X}) $ and $ \mathbb{K}(\mathbb{X},\mathbb{Y}) = \mathbb{K}(\mathbb{X}), $ if $ \mathbb{X} = \mathbb{Y}. $ \\

\noindent For any two elements $ x,y \in \mathbb{X}, $ $ x $ is said to be orthogonal to $ y $ in the sense of Birkhoff-James, written as $ x \perp_B y, $ if 
\[ \|x\| \leq \|x+\lambda y\| ~ \mbox{for all}~ \lambda \in \mathbb R. \]
  Likewise, for any two elements $ T,A \in  \mathbb{L}(\mathbb{X}), $ $ T $ is said to be orthogonal to $ A $ in the sense of Birkhoff-James, written as $ T \perp_B A, $ if 
  \[ \|T\| \leq \|T+ \lambda A\| ~\mbox{for all}~ \lambda \in \mathbb R. \]
	\noindent We refer the readers to the classic works \cite{B,Ja,Jb}, for more information on Birkhoff-James orthogonality.
	\noindent For a bounded linear operator $ T \in \mathbb{L}(\mathbb{X}), $ let $ M_T $ denote the collection of all unit vectors in $ \mathbb{X} $ at which $ T $ attains norm, i.e., 
\[ M_T = \{  x \in S_{\mathbb{X}} : \| Tx \| = \| T \|  \}. \]

\noindent In this paper, given $ T \in \mathbb{L}(\mathbb{X}, \mathbb{Y}), $ we first obtain a necessary condition for $ x \in S_{\mathbb{X}} $ to be such that $ x \in M_T. $ The condition can be expressed in a particularly convenient form, if the Banach spaces $ \mathbb{X}, \mathbb{Y} $ are smooth. As it turns out, a particular concept, introduced in \cite{Sb} plays a very significant role in the whole scheme of things. For the sake of completeness, let us mention the relevant definitions here : \\
For any two elements $ x, y $ in $ \mathbb{X}, $ let us say that $ y \in x^{+} $ if $ \| x + \lambda y \| \geq \| x \| $ for all $ \lambda \geq 0. $ Accordingly, we say that $ y \in x^{-} $ if $ \| x + \lambda y \| \geq \| x \| $ for all $ \lambda \leq 0. $ Basic properties related to this notion have been explored in Proposition $ 2.1 $ of \cite{Sb}. Let $ x^{\perp} = \{ y \in \mathbb{X} : x \perp_{B} y \}. $ Using this concept, we obtain a necessary condition for $ T \in \mathbb{L}(\mathbb{X}) $ to attain norm at $ x \in S_{\mathbb{X}}. $ We prove four corollaries to our main result, Theorem $ 2.2, $ in order to illustrate its importance and strength. Let us recall the relevant definitions in this context. \\
For an element $ x \in  \mathbb{X}, $ let us say that $ x $ is left symmetric (with respect to Birkhoff-James orthogonality) if $ x \perp_{B} y $ implies $ y \perp_{B} x $ for any $ y \in \mathbb{X}. $ Similarly, let us say that $ x $ is right symmetric (with respect to Birkhoff-James orthogonality) if $ y \perp_{B} x $ implies $ x \perp_{B} y $ for any $ y \in \mathbb{X}. $ $ T \in \mathbb{L}(\mathbb{X}) $ is said to satisfy the Daugavet equation if $ \| I+T \| = 1 + \| T \|, $ where $ I $ is the identity operator on $ \mathbb{X}. $ Using these concepts, we apply Theorem $ 2.2 $ to obtain various interesting properties of a bounded linear operator on a smooth Banach space. First, we obtain an expression for the kernel of a  bounded linear operator defined on a Banach space. Next, we prove that if the underlying Banach space is finite dimensional, strictly convex and smooth, then every linear operator, satisfying the Daugavet equation, must have an invariant subspace of codimension one. We also prove that in a finite dimensional smooth Banach space $ \mathbb{X} $, if the kernel of a nonzero bounded linear operator $ T $ contains a nonzero right symmetric point, then $ span~M_T $ is a proper subspace of $ \mathbb{X}. $ As the last corollary, we prove that in smooth Banach spaces, image of a left symmetric point under an isometry must be a left symmetric point.\\
It is easy to observe that $ T \in \mathbb{L}(\mathbb{X}) $ is a scalar multiple of an isometry if and only if $ M_T  = S_{\mathbb{X}}. $ It was proved in \cite{K} that  a norm one linear operator $ T \in \mathbb{L}(\mathbb{X}) $ is an isometry if and only if $ T $ preserves Birkhoff-James orthogonality, i.e., $ x \perp_B y \Rightarrow Tx \perp_B Ty. $ Motivated by this characterization of isometries on Banach spaces, it is natural to ask whether every bounded linear operator must preserve Birkhoff-James orthogonality at some point of the unit sphere. It was indeed a pleasant surprise to us that this question can be answered in the affirmative for compact linear operators defined on reflexive and smooth Banach spaces. In this connection, we also obtain a characterization of smooth Banach spaces in terms of operator norm attainment and Birkhoff-James orthogonality.\\
As another potential application of Theorem $ 2.2, $ we explore the possible norm attainment set of a bounded linear operator defined on $ l_{p}^{2} $ spaces. For a bounded subset $ A $ of a Banach space $ \mathbb{X}, $ let $ | A | $ denote the cardinality of $ A. $ If $ A $ is finite then $ | A | $ is the number of elements in $ A. $ Let $ T \in \mathbb{L}(l_{p}^{2}) (p \in \mathbb{N}\setminus \{ 1 \}) $ be such that $ T $ is not a scalar multiple of an isometry on $ l_{p}^{2}. $ Using Theorem $ 2.2, $ we prove that $ | M_T | \leq 2(8p - 5). $ To the best of our knowledge, such an estimation is being presented for the very first time. It should be noted that our estimation may not be optimal and there remains the scope to obtain better estimations by using other geometric and analytic arguments. Moreover, we strongly expect that such a result would open up the possibilities of obtaining analogous results for bounded linear operators defined on classical Banach spaces of higher dimensions.

\section{ Main Results.}

Let us begin this section with an easy proposition.

\begin{prop}
Let $ \mathbb{X}, \mathbb{Y} $ be Banach spaces, $ T \in \mathbb{L}(\mathbb{X}, \mathbb{Y}) $ and $ x \in M_T.  $ Then for any $ y \in \mathbb{X}, Tx \perp_{B} Ty \Rightarrow x \perp_{B} y. $  
\end{prop}
\begin{proof}
If possible suppose that $ x \not\perp_{B} y. $ Then there exists a nonzero scalar $ \lambda_{0} \in \mathbb{R} $ such that $ \| x+\lambda_0 y \| < \| x \| = 1. $ Without loss of generality, we may and do assume that $ \lambda_0 < 0. $
Consider the element $ z = \frac{x+\lambda_0 y}{\| x+\lambda_0 y \|}. $ Clearly, $ z = \alpha x + \beta y, $ where $ \alpha = \frac{1}{\| x+\lambda_0 y \|} > 1, \beta = \frac{\lambda_0}{\| x+\lambda_0 y \|} < 0. $ Now, we have,\\
\noindent $ \| Tz \| = \| \alpha Tx + \beta Ty \| = |\alpha| \|  Tx + \frac{\beta}{\alpha} Ty \| > \| Tx \| = \| T \|, $ since $ |\alpha| > 1, Tx \perp_B Ty, x \in M_T. $ However, this is clearly a contradiction, since $ \| z \| = 1. $ This completes the proof of the proposition. 
\end{proof}

In fact, the converse to Proposition $ 2.1 $ is also true if both $ \mathbb{X}, \mathbb{Y} $ are smooth Banach spaces. To this end, let us prove a more general result that also gives a necessary condition for the norm attainment of a bounded linear operator on a Banach space, at a particular point of the unit sphere.

\begin{theorem}\label{theorem:preserve}
Let $ \mathbb{X}, \mathbb{Y} $ be Banach spaces, $ T \in \mathbb{L}(\mathbb{X}, \mathbb{Y}) $ and $ x \in M_T.  $ Then \\
(i) $ T(x^{+} \setminus x^{\perp}) \subseteq (Tx)^{+} \setminus (Tx)^{\perp}, $ \\
(ii) $ T(x^{-} \setminus x^{\perp}) \subseteq (Tx)^{-} \setminus (Tx)^{\perp}. $ \\
If in addition, both $ \mathbb{X} $ and $ \mathbb{Y} $ are smooth then $ T(x^{\perp}) \subseteq (Tx)^{\perp}. $
\end{theorem}

\begin{proof}
Let us first prove (i). Let $ u \in x^{+} \setminus x^{\perp} $ be chosen arbitrarily. Since $ x \not\perp_{B} u, $ there exists a nonzero scalar $ \lambda_0  $ such that $ \| x+\lambda_0 u \| < \| x \| = 1. $ Since $ u \in x^{+}, $ we must have $ \lambda_0 < 0. $ As in the proof of proposition $ 2.1, $ consider the element $ z = \frac{x+\lambda_0 u}{\| x+\lambda_0 u \|} = \alpha x + \beta u, $ where $ \alpha = \frac{1}{\| x+\lambda_0 u \|} > 1, \beta = \frac{\lambda_0}{\| x+\lambda_0 u \|} < 0. $ Now, we have, \\
\noindent $ \| Tz \| = \| \alpha Tx + \beta Tu \| = |\alpha| \|  Tx + \frac{\beta}{\alpha} Tu \| > \| Tx + \frac{\beta}{\alpha} Tu \|, $ since $ |\alpha| > 1. $ \\
We claim that $ Tu \notin (Tx)^{-}. $  Suppose, $ Tu \in (Tx)^{-}. $ Since $ \frac{\beta}{\alpha} < 0, $ it follows from Proposition $ 2.1 $ of \cite{Sb} that $ \frac{\beta}{\alpha} Tu \in (Tx)^{+}. $ Therefore, we have, \\
\noindent $ \| Tz \| > \| Tx + \frac{\beta}{\alpha} Tu \| \geq \| Tx \| = \| T \|, $ since $ x \in M_{T}. $ Clearly, this is a contradiction as $ \| z \| = 1. $ Thus, we must have, $ Tu \notin (Tx)^{-}. $ It now follows from Proposition $ 2.1 $ of \cite{Sb} that $ Tu \in (Tx)^{+} \setminus (Tx)^{\perp}. $ Since $ u \in x^{+} \setminus x^{\perp} $ was chosen arbitrarily, this completes the proof of (i). \\
\noindent The proof of (ii) can now be completed using similar arguments. \\
\noindent Let us now assume that in addition, $ \mathbb{X}, \mathbb{Y} $ are smooth Banach spaces. Then there exists a unique hyperplane of support $ x+H_0 $ to $ B_{\mathbb{X}} $ at $ x, $ where $ H_0 $ is a subspace of $ \mathbb{X}, $ having codimension one. Clearly, $ H_0 \equiv x^{\perp}. $ $ H_0 $ divides $ \mathbb{X} $ into two closed half-planes whose intersection is $ H_0. $ Let $ H_1 $ denote the closed half-plane containing $ x $ and let $ H_2 $ denote the other closed half-plane. Then it is easy to see that $ H_1 \equiv x^{+} $ and $ H_2 \equiv x^{-}. $ Furthermore, every element of $ H_0 $ can be approximated by elements exclusively from either of the sets $ H_1 \setminus H_0 $ and $ H_2 \setminus H_0. $ \\
\noindent Let $ w \in x^{\perp} $ be arbitrary. Let $ \{y_n\} $ be a sequence in $ H_1 \setminus H_0 $ such that $ y_n \longrightarrow w. $ It now follows from (i) and the respective identifications of $ x^{\perp}, x^{+} $ with $ H_0, H_1 $ that $ T(y_n) \in (Tx)^{+} \setminus (Tx)^{\perp}. $ Since $ T $ is continuous and $ y_n \longrightarrow w, $ we must have, $ Tw \in \overline{(Tx)^{+} \setminus (Tx)^{\perp}}. $ Now, considering a sequence in $ H_2 \setminus H_0 $ that converges to $ w $ and using similar arguments, it is easy to show that $ Tw \in \overline{(Tx)^{-} \setminus (Tx)^{\perp}}. $ Thus, we have, \\
\[Tw \in \overline{(Tx)^{+} \setminus (Tx)^{\perp}} \cap \overline{(Tx)^{-} \setminus (Tx)^{\perp}}.\]

\noindent However, since $ \mathbb{Y} $ is smooth, $ \overline{(Tx)^{+} \setminus (Tx)^{\perp}} \cap \overline{(Tx)^{-} \setminus (Tx)^{\perp}} = (Tx)^{\perp}. $ This proves that $ Tw \in (Tx)^{\perp} $ for each $ w \in x^{\perp}, $ thereby completing the proof of the theorem. 
\end{proof}

We now obtain a number of corollaries to our main result, Theorem $ 2.2, $ in order to establish its significance. We would also like to remark that the varying applications of Theorem $ 2.2 $ betray its strength and importance in the study of the geometry of Banach spaces. First, we obtain an expression for the kernel of a  bounded linear operator defined on a Banach space, in terms of norm attainment. 

\begin{cor}
Let $ \mathbb{X} $ be a Banach space and let $ T \in \mathbb{L}(\mathbb{X}). $ Then
\[  ker~T \subseteq \bigcap_{x \in M_T} x^{\perp}. \]
In particular, if $ \mathbb{X} $ is a two dimensional smooth and strictly convex Banach space, then any linear operator attaining norm at more than one pair of points must be invertible.
\end{cor}

\begin{proof}
We first note that if $ T $ does not attain norm, i.e., $ M_T = \emptyset, $ then the first part of the theorem follows trivially. Let us assume that $ M_T \neq \emptyset. $ Let $ z \in ker~T $ be chosen arbitrarily. For any $ x \in M_T, $ we have,  $ Tz \in (Tx)^{\perp}, $ since $ Tz = 0. $ Theorem $ 2.2 $ implies that $ z \notin x^{+} \setminus x^{\perp} $ and $ z \notin x^{-} \setminus x^{\perp}. $ Since $ \mathbb{X} = (x^{+} \setminus x^{\perp}) \cup (x^{-} \setminus x^{\perp}) \cup x^{\perp}, $ it now follows that, $ z \in x^{\perp}. $ As this is true for every $ z \in ker~T $ and for every $ x \in M_T, $ we must have : 
\[ ker~T \subseteq \bigcap_{x \in M_T} x^{\perp}. \]

Let us now assume that $ \mathbb{X} $ is a two dimensional smooth and strictly convex Banach space and let $ T \in \mathbb{L}(\mathbb{X}). $ Let $ x_1, x_2 \in M_T $ be such that $ x_1 \neq \pm x_2. $ We first claim that $ x_{1}^{\perp} \cap x_{2}^{\perp} = \{0\}. $ Let $ z \in x_{1}^{\perp} \cap x_{2}^{\perp}. $ Since $ \mathbb{X} $ is a two dimensional strictly convex Banach space, Birkhoff-James orthogonality is left symmetric in $ \mathbb{X}. $ Therefore, we have : 
\[ (\alpha x_1 + \beta x_2) \perp_{B} z, \text{for any scalars $ \alpha, \beta. $ } \]
In particular, this implies that $ z \perp_{B} z. $ Therefore, we must have, $ z = 0. $ Now, the second part of the corollary follows directly from the first part, since $ ker~T \subseteq (x_{1}^{\perp} \cap x_{2}^{\perp}) = \{0\} \subseteq ker~T. $ 
\end{proof}

In the next corollary, we prove that any linear operator defined on a finite dimensional smooth and strictly convex Banach space $  \mathbb{X} $, that satisfies the Daugavet equation, must have an invariant subspace of codimension one.

\begin{cor}
Let $ \mathbb{X} $ be a finite dimensional smooth and strictly convex Banach space and let $ T \in \mathbb{L}(\mathbb{X}) $ satisfies the Daugavet equation $ \|I + T\| = 1 + \| T \|. $ Then $ \mathbb{X} $ has a $ T- $invariant subspace of codimension one.
\end{cor}

\begin{proof}
We first note that if $ T $ is the zero operator then the theorem is trivially true. Let us assume that $ T $ is nonzero. We also observe from Lemma $ 2.1 $ of \cite{A} that for any nonzero bounded linear operator $ T $ on $ \mathbb{X}, $ $ T $ satisfies the Daugavet equation if and only if $ \alpha T $ satisfies the Daugavet equation, where $ \alpha > 0 $ is any scalar. Thus, without loss of generality, we may and do assume that $ \| T \| = 1. $ The next thing to observe is that since $ \mathbb{X} $ is finite dimensional, there exists a unit vector $ x_0 $ such that $ \| I + T \| = \| (I + T) x_0 \|. $ We claim that $ Tx_0 = x_0. $ Indeed,
\[ 2 = \| I+T \| = \| (I + T) x_0 \| = \| x_0 + Tx_0 \| \leq \| x_0 \| + \| Tx_0 \| \leq 1 + \| T \| = 2. \]
Since $ \mathbb{X} $ is strictly convex, $ \| x_0 + Tx_0 \| < \| x_0 \| + \| Tx_0 \|, $ if $ Tx_0 \notin \{kx_0 : k \geq 0\}. $ Therefore, we must have, $ Tx_0 = k_0x_0, $ for some $ k_0 \geq 0. $  On the other hand, $ 2 = \| x_0 + Tx_0 \| \leq \| x_0 \| + \| Tx_0 \| = 1 + \| Tx_0 \| \leq 2 $ implies that $ k_0 = \| Tx_0 \| = 1. $ This proves our claim.\\
Thus, we have, $ 1 = \| T \| = \| Tx_0 \|. $ This proves that $ x_0 \in M_T. $ Since $ \mathbb{X} $ is smooth and $ x_0 \in M_T, $ applying Theorem $ 2.2, $ we have, 
\[ T(x_{0}^{\perp}) \subseteq (Tx_0)^{\perp} = x_{0}^{\perp}. \]
Since $ \mathbb{X} $ is smooth, $ x_{0}^{\perp} $ is a subspace of $ \mathbb{X}, $ having codimension one. Thus, $ x_{0}^{\perp} $ is a $ T- $invariant subspace of $ \mathbb{X}, $ having codimension one.
\end{proof}

\begin{remark}
In particular, it follows from the method used in the proof of Corollary $ 2.2.2 $ that if $ T $ is a norm one linear operator on a finite dimensional strictly convex and smooth Banach space $ \mathbb{X} $ such that $ T $ satisfies the Daugavet equation then $ T $ has a fixed point on $ S_{\mathbb{X}} $.
\end{remark}

Next, we prove that in a finite dimensional smooth Banach space $ \mathbb{X} $, if the kernel of a nonzero bounded linear operator $ T $ contains a nonzero right symmetric point, then $ span~M_T $ is a proper subspace of $ \mathbb{X}. $ 

\begin{cor}
Let $ \mathbb{X} $ be a finite dimensional smooth  Banach space and let $ T \in \mathbb{L}(\mathbb{X}) $ be a nonzero bounded linear operator such that $ ker~T $ contains a nonzero right symmetric point. Then $ span~M_T $ is a proper subspace of $ \mathbb{X}. $ In particular, if $ T \in \mathbb{L}(l_{p}^{n}) $ is such that for some $ i \in \{1, 2, \ldots, n\}, T(e_i) = 0, $ where $ e_i $ denotes the unit vector whose $ i- $th component is $ 1 $ and all other components are $ 0, $ then 
\[min~\{\|Tx_j\| : j \in \{1, 2, \ldots, n\}\} < \| T \|, \text{for any basis $ \{ x_1, x_2, \ldots, x_n \} $ of $ l_{p}^{n}. $ }\]  
\end{cor}

\begin{proof}
Suppose, on the contrary, $ span~M_T = \mathbb{X}. $ Then $ M_T $ contains a basis $ \{ x_1, x_2, \ldots, x_n \} $ of $ \mathbb{X}. $ Let $ z \in ker~T $ be a nonzero right symmetric point. Without loss of generality, we may and do assume that $ \| z \| = 1. $ \\
Let $ z = \sum_{i=1}^{i=n} c_ix_i, $ for some scalars $ c_1,c_2,\ldots, c_n. $ Applying Corollary $ 2.2.1, $ we see that, $ x_i \perp_B z, $ for each $ i = 1, 2, \ldots, n. $ Since $ z $ is right symmetric, we conclude that $ z \perp_B x_i, $ for each $ i = 1, 2, \ldots, n. $ As $ \mathbb{X} $ is smooth, Birkhoff-James orthogonality is right additive in $ \mathbb{X}. $ Therefore, we must have, $ z \perp_B \sum_{i=1}^{i=n} c_ix_i=z. $ This implies that $ z = 0, $ a contradiction to our assumption that $ z $ is nonzero. This contradiction completes the proof of the first part of the corollary. \\
The second part of the corollary now follows directly from the first part by observing that each $ e_i $ is a right symmetric point in $ l_{p}^{n}. $
\end{proof}

As the final corollary, we prove that in smooth Banach spaces, image of a left symmetric point under an isometry must be a left symmetric point.

\begin{cor}
Let $ \mathbb{X} $ be a smooth  Banach space and let $ T \in \mathbb{L}(\mathbb{X}) $ be an isometry.  If $ x \in \mathbb{X} $ is a left symmetric point then $ Tx $ is also a left symmetric point.
\end{cor}

\begin{proof}
We first note that $ M_T = S_{\mathbb{X}}, $ as $ T $ is an isometry. Let $ Tx \perp_B y, $ for some $ y \in S_{\mathbb{X}}. $ If $ y = 0 $ then $ y \perp_B Tx. $ Let $ y \neq 0. $ Since $ T $ is an isometry, $ T $ is invertible. Let $ z (\neq 0) \in \mathbb{X} $ be such that $ y = Tz. $ Since $ Tx \perp_B Tz $ and $ x \in M_T, $ applying Proposition $ 2.1, $ we have, $ x \perp_B z. $ This implies that $ z \perp_B x, $ as $ x $ is a left symmetric point in $ \mathbb{X}. $ Since $ \mathbb{X} $ is smooth and $ \frac{z}{\|z\|} \in M_T, $ applying Theorem $ 2.2, $ we have, $ T(\frac{z}{\|z\|}) \perp_B Tx. $ Using the homogeneity property of Birkhoff-James orthogonality, it is now easy to see that $ y=Tz\perp_B Tx. $ This completes the proof of the corollary.
\end{proof}

It is interesting to observe that the smoothness condition in the last part of Theorem $ 2.2 $ is indeed required. We give the next example to illustrate our point.
\begin{example}
Consider the linear operator $ T \in \mathbb{L}(l_{2}^{\infty}) $ given by $ T(1, 1) = (1, 0) $ and $ T(1, -1) = (\frac{1}{2}, \frac{1}{2} ). $ Then it is easy to check that $ (1, 1) \in M_T. $ We also note that $ (1, 1) \perp_{B} (-\frac{1}{2}, 1). $ However, we have, $ T(-\frac{1}{2}, 1) = (-\frac{1}{8}, -\frac{3}{8}) \notin (T(1, 1))^{\perp} = (1, 0)^{\perp} = \{ (0, \beta) : \beta \in \mathbb{R} \}. $
\end{example}

Combining Proposition $ 2.1 $ and Theorem $ 2.2, $ we have the following simple and useful necessary condition for the norm attainment of a bounded linear operator on a smooth Banach space, at a particular point of the unit sphere : 

\begin{theorem}\label{theorem:necessary}
Let $ \mathbb{X} $ be a smooth Banach space, $ T \in \mathbb{L}(\mathbb{X}) $ and $ x \in M_T.  $ Then for any $ y \in \mathbb{X}, x \perp_{B} y \Leftrightarrow Tx \perp_{B} Ty.$
\end{theorem}

An isometry on $ \mathbb{X} $ preserves Birkhoff-James orthogonality at every point of $ \mathbb{X} $ \cite{K}. The next theorem may be regarded as a ``local" version of this fact, which is valid for any compact linear operator defined on a reflexive and smooth Banach space. Apart from applying Theorem $ 2.2, $ the only thing that we need to observe is that a compact linear operator on a reflexive Banach space must attain norm.

\begin{theorem}\label{theorem:compact}
Let $ \mathbb{X} $ be a reflexive, smooth Banach space. Let $ T \in \mathbb{K}(\mathbb{X}). $  Then there exists $ x \in S_{\mathbb{X}} $ such that $ T $ preserves Birkhoff-James orthogonality at $ x $ i.e., for any $ y \in \mathbb{X}, x \perp_{B} y \Leftrightarrow Tx \perp_{B} Ty.$ 
\end{theorem}

In Example $ 2.2.1, $ we have illustrated the fact that the smoothness condition in the last part of Theorem $ 2.2 $ cannot be relaxed. In fact, in this context it is possible to obtain a nice geometric characterization of smooth Banach spaces in terms of operator norm attainment and Birkhoff-James orthogonality. To this end, we first prove the following theorem :

\begin{theorem}\label{theorem:sufficient}
Let $ \mathbb{X} $ be a Banach space. If for every $ T \in \mathbb{L}(\mathbb{X}) $ and for every $ x \in M_T, $ we have, $ T(x^{\perp}) \subseteq (Tx)^{\perp}, $ then $ \mathbb{X} $ is smooth.   
\end{theorem}

\begin{proof}
$ \mathbb{X} $ is smooth if and only if given any $ x \in S_{\mathbb{X}}, $ there exists a unique hyperplane $ H_x $  of codimension one such that $ x \perp_{B} H_x. $ Suppose on the contrary, there exists a $ x_0 \in S_{\mathbb{X}} $ such that $ x_0 \perp_{B} H_{x_{1}} $ and $ x_0 \perp_{B} H_{x_{2}}, $ where $ H_{x_{1}}, H_{x_{2}} $ are two different hyperplanes of codimension one. Clearly, any element $ z  \in \mathbb{X} $ can be written as $ z = \alpha x_0 + h, $ where $ \alpha \in \mathbb{R}, h \in H_{x_{1}}. $ Let us define an operator $ T  $ on $ \mathbb{X} $ in the following way :
\[T(\alpha x_0 + h) = \alpha x_0, \text{for each}~ \alpha \in \mathbb{R}~ \text{and for each}~ h \in H_{x_{1}}. \]

Clearly, $ T $ is well-defined and linear. Since $ x_0 \perp_{B} H_{x_{1}}, $ it is easy to check that $ T $ is bounded and $ x_0 \in M_T. $ Let us now choose a $ y \in H_{x_{2}} \setminus H_{x_{1}}. $ Let $ y = \alpha_0 x_0 + h_0, $ where $ \alpha_0 \in \mathbb{R}  $ and $ h_0 \in H_{x_{1}}. $ Since $ y \in H_{x_{2}} \setminus H_{x_{1}}, $ we must have, $ y \in x_{0}^{\perp} ~\text{and}~ \alpha_0 \neq 0. $ Therefore, according to the condition stated in the theorem, we have :
\[Tx_0 = x_0 \perp_{B} Ty = \alpha_0 x_0.\]
However, this clearly leads to a  contradiction as $ x_0 \in S_{\mathbb{X}} $ and $ \alpha_{0} \neq 0. $ This completes the proof of the theorem.
\end{proof}
Thus, we have the following characterization of smooth Banach spaces :

\begin{theorem}\label{theorem:characterization}
A Banach space $ \mathbb{X} $ is smooth if and only if for every $ T \in \mathbb{L}(\mathbb{X}) $ and for every $ x \in M_T, $ we have, $ T(x^{\perp}) \subseteq (Tx)^{\perp}. $
\end{theorem}

As another potential application of Theorem $ 2.2, $ let us now explore the possible norm attainment set of a linear operator $ T $ on $ \mathbb{X} = l_{p}^{2} (p \in \mathbb{N}\setminus \{ 1 \})$ spaces. If $ T $ is a scalar multiple of an isometry, then evidently $ M_T = S_{\mathbb{X}}. $ So let us restrict our attention to linear operators that are not scalar multiples of an isometry. It seems natural to ask what we can say about $ | M_T |, $ in that case. We once again apply Theorem $ 2.2 $ to answer this question. 

\begin{theorem}\label{theorem:p-spaces}
Let $ \mathbb{X} = l_{p}^{2} (p \in \mathbb{N}\setminus \{ 1 \}) $ and let $ T \in \mathbb{L}(\mathbb{X}) $ be such that $ T $ is not a scalar multiple of an isometry. Then $ | M_T | \leq 2(8p - 5). $
\end{theorem}

\begin{proof}
Without loss of generality, we may and do assume that $ \| T \| = 1. $ We first note that if $ | M_T | \leq 6 $ then we have nothing to prove, since $ \| M_T \| \leq 6 < 2(8p-5) $ for each $ p \in \mathbb{N} \setminus \{ 1 \}. $ Let $ | M_T | > 6. $ Since $ \mathbb{X} $ is a two dimensional smooth and strictly convex Banach space, it follows from Corollary $ 2.2.1 $ that $ T $ must be invertible. In particular, $ T(x_1, y_1) = T(x_2, y_2) = \pm (1, 0) $ implies that $ (x_1, y_1) = \pm (x_2, y_2). $ Since $ | M_T | > 6, $ we can choose $ (x, y) \in M_T $ such that $ (x, y) \neq \pm (1, 0), \pm (0, 1) $ and $ T(x, y) \neq \pm (1, 0). $\\
Let $ T =
 	\begin{pmatrix}
  		a & b \\ 
  		\rule{0em}{3ex}c & \hphantom{-} d
 	\end{pmatrix} $ be the matrix representation of $ T $ with respect to the standard ordered basis of $ l_{p}^
	{2}. $ Since $ (x, y) \neq \pm (0, 1), $ there exists $ k \in \mathbb{R} $ such that $ y = kx. $ Since $ (x, y) \neq \pm (1, 0), $ we have, $ k \neq 0. $ Using elementary calculus, it is easy to check that either $ (x, y) \perp_B (1, -\frac{x^{p-1}}{y^{p-1}}) $ or $ (x, y) \perp_B (1, \frac{(-x)^{p-1}}{y^{p-1}}). $ Let us first assume that $ (x, y) \perp_B (1, -\frac{x^{p-1}}{y^{p-1}}). $ Since $ (x, y) \in M_T, $ applying Theorem $ 2.2, $ we have, 
\[ T(x, y) \perp_B T(1, -\frac{x^{p-1}}{y^{p-1}}) ~\text{i.e.,}~(ax+by, cx+dy) \perp_B (a-b\frac{x^{p-1}}{y^{p-1}}, c-d\frac{x^{p-1}}{y^{p-1}}). \]

We note that $ (ax+by, cx+dy) = T(x, y)  \neq \pm (1, 0). $ Therefore, if $ (ax+by, cx+dy) \perp_B (z, lz), $ where $ z, l \in \mathbb{R}, $ then either $ l=-\frac{(ax+by)^{p-1}}{(cx+dy)^{p-1}} $ or $ l=\frac{(-(ax+by))^{p-1}}{(cx+dy)^{p-1}}. $ Let us first assume that $ l=-\frac{(ax+by)^{p-1}}{(cx+dy)^{p-1}}. $ \\

We claim that $ T(1, -\frac{x^{p-1}}{y^{p-1}}) \neq (0, \gamma) $ for any $ \gamma \in \mathbb{R}. $ Suppose, on the contrary, $ T(1, -\frac{x^{p-1}}{y^{p-1}}) = (0, \gamma) $ for some $ \gamma \in \mathbb{R}. $ Since $ (x, y) \perp_{B} (1, -\frac{x^{p-1}}{y^{p-1}}) $ and $ (x, y) \in M_T, $ it follows from Theorem $ 2.2 $ that $ T(x, y) \perp_{B} T(1, -\frac{x^{p-1}}{y^{p-1}}), $ i.e., $ T(x, y) \perp_{B} (0, \gamma). $ Since $ \| T \| = 1 $ and $ (x, y) \in M_T, $ we must have $ T(x, y) = \pm (1, 0), $ a contradiction to our initial choice of $ (x, y). $ This proves our claim.\\
Therefore, it is legitimate to write the following equation : 
\[ -\frac{(ax+by)^{p-1}}{(cx+dy)^{p-1}} = \frac{c-d\frac{x^{p-1}}{y^{p-1}}}{a-b\frac{x^{p-1}}{y^{p-1}}}. \]

Recalling that $ y=kx, $ where $ k \neq 0, $ the above equation reduces to : 
\[ \frac{(a+bk)^{p-1}}{(c+dk)^{p-1}} = \frac{-ck^{p-1}+d}{ak^{p-1}-b}. \] 
It is easy to see that the above equation can be written in the form of a polynomial equation in $ k, $ of degree at most $ 2p-2, $ with the coefficients coming from $ \mathbb{R}. $ Therefore, this equation can have at most $ 2p-2 $ number of different solutions for $ k. $ It is easy to see that each real value of $ k $ gives rise to exactly two points on the unit sphere $ S_{\mathbb{X}}, $ which are antipodal. It is at this point of the proof that we remember that we have chosen $ (x,y) \in M_T $ such that $ (x, y) \neq \pm (1, 0), \pm (0, 1), T(x, y) \neq \pm (1, 0) $ and have also ignored some possible cases along the course of the proof.  Therefore, taking all possible cases into consideration, we must have :
\[ | M_T | \leq 8(2p-2)+6=2(8p-5). \] 

\end{proof}

\begin{remark}
We do not expect the upper bound $ 2(8p-5) $ obtained in Theorem $ 2.7 $ to be an optimal upper bound for $ | M_T |. $ In fact, it seems to us that using stronger geometric and algebraic techniques, it might be possible to improve upon the bound $ 2(8p-5). $ In view of this, we pose the following open question : 
\end{remark}

\noindent \textbf{Open Question :} For the Banach space $ \mathbb{X} = l_{p}^{2}, $ find an optimal bound for $ |M_T|, $ where $ T \in \mathbb{L}(\mathbb{X}) $ is not a scalar multiple of an isometry.

\begin{remark}
The case $ p=2 $ deserves special mention, since in this case $ \mathbb{X} $ is the two dimensional Euclidean space. It follows from Theorem $ 2.2 $ of \cite{Sa} that if $ T \in \mathbb{L}(l_{2}^{2}) $ then either $ M_T = S_{\mathbb{X}} $  or $ M_T $ is a doubleton.  In the first case, $ T $ is a scalar multiple of an isometry and in the second case, $ T $ is smooth. 
\end{remark}

\begin{remark}
We would like to comment that in light of the results obtained in the present paper, finding a characterization of the norm attainment set of a bounded linear operator defined on a Banach space turns out to be a particularly significant and nontrivial question in the study of the geometry of Banach spaces. It is not difficult to obtain a sufficient condition for a bounded linear operator $ T $ on a smooth Banach space $ \mathbb{X} $ to attain norm at $ x \in S_{\mathbb{X}}. $ In fact, $ Tx=x, T|_{x^{\perp}} = \alpha I, $ for some $ \alpha \in (0, 1), $ is one such condition. However, such a condition is certainly not necessary. It is perhaps befitting that we end the present paper with the following open question :
\end{remark}

\noindent \textbf{Open Question :} Let $ T $ be a bounded linear operator defined on a Banach space $ \mathbb{X}. $ Find a necessary and sufficient condition for $ x \in S_{\mathbb{X}} $ to be such that $ x \in M_T. $

\bibliographystyle{amsplain}

\end{document}